\documentclass[10pt,a4paper]{amsart}
\usepackage{amsmath,amsthm,amssymb,latexsym}
\usepackage{enumerate}
\usepackage{graphicx}

\newtheorem{theorem}{Theorem}
\newtheorem{lemma}[theorem]{Lemma}

\newtheorem{exa}[theorem]{Example}
\newtheorem{rema}[theorem]{Remark}

\newtheorem{defi}[theorem]{Definition}

\newtheorem{pro}[theorem]{Open Problem}

\numberwithin{equation}{section} \numberwithin{theorem}{section}

\newcommand{\C}{\mathbf{C^{\ast}}}

\title{When are dual Cayley automaton semigroups finite?}
\author{Victor Maltcev}

\date{}

\begin{document}

\maketitle

\begin{center}
Mathematical Institute, University of St Andrews\\
St Andrews, Fife KY16 9SS, Scotland,
\texttt{victor.maltcev@gmail.com}
\end{center}

\begin{abstract}
In this note we prove that, for a finite semigroup $S$, the dual
Cayley automaton semigroup $\mathbf{C^{\ast}}(S)$ is finite if and
only if $S$ is $\mathcal{H}$-trivial and has no non-trivial right
zero subsemigroups.
\end{abstract}



\section{Introduction \& Main Theorem}

In a recent paper~\cite{SS}, Silva \& Steinberg for any finite
semigroup $S$ define the \emph{Cayley automaton $\mathcal{C}(S)$}
associated with $S$: its state set is $S$; the alphabet, on which
the states act, is $S$; and when $\mathcal{C}(S)$ is in state $s$
and reads symbol $x$, it moves to the state $sx$ and outputs the
symbol $sx$. In other words, this is the automaton obtained from the
Cayley graph of $S$, specifying the output symbol on the arc, going
from $s$ and labeled by $x$, to be $sx$.

As to every automaton, to the automaton $\mathcal{C}(S)$ we can
associate its automaton semigroup $\mathbf{C}(S)$, i.e. the
transformation semigroup on infinite sequences $S^{\infty}$,
generated by states $s\in S$ viewed as the correspondent
transformations of $S^{\infty}$. In~\cite{SS} Silva \& Steinberg
prove the following

\begin{theorem}\label{th:Cayley-groups}
Let $G$ be a finite non-trivial group. Then $\mathbf{C}(G)$ is a
free semigroup of rank $|G|$.
\end{theorem}

The author and Alan Cain after reading~\cite{SS} decided to study
semigroups $\mathbf{C}(S)$ in more details. In~\cite{CM} we
characterized when $\mathbf{C}(S)$ is free, commutative, trivial, a
left or a right zero semigroup. But the most interesting result we
obtain in~\cite{CM} is:

\begin{theorem}\label{th:ours}
Let $S$ be a finite semigroup. Then $\mathbf{C}(S)$ is finite if and
only if $S$ is $\mathcal{H}$-trivial.
\end{theorem}

The sufficiency of Theorem~\ref{th:ours} was proved in a recent
work~\cite{M}. While the author and Alan Cain where
preparing~\cite{CM}, we did not know about the work~\cite{M}. The
sufficiency of Theorem~\ref{th:ours} is the harder bit of this
result. We provided in~\cite{CM} two different proofs of sufficiency
of Theorem~\ref{th:ours}: using the method of actions on sequences
and using the method of wreath recursions. These two approaches and
the method from~\cite{M} are all different in their realizations but
bear the same spirit: the aim was to find an ultimate constant $N$
such that any product of states from $\mathcal{C}(S)$ of length $N$
can be reduced to a shorter word.

Now, if in the Cayley graph of a finite semigroup $S$, instead of
the output symbol $sx$ we put the symbol $xs$, we arrive at the
\emph{dual Cayley automaton $\mathcal{C}^{\ast}(S)$}: the state and
alphabet sets being $S$ and when $\mathcal{C}^{\ast}(S)$ is in state
$s$ and reads symbol $x$, it moves to the state $sx$ and outputs the
symbol $xs$.

Denote by $\mathbf{C^{\ast}}(S)$ the automaton semigroup generated
by the automaton $\mathcal{C}^{\ast}(S)$. The main result of the
note is

\begin{theorem}\label{th:dual-cute}
Let $S$ be a finite semigroup. Then $\mathbf{C^{\ast}}(S)$ is finite
if and only if $S$ is $\mathcal{H}$-trivial and does not contain
non-trivial right zero subsemigroups.
\end{theorem}

We prove Theorem~\ref{th:dual-cute} in Section~\ref{sec:finiteness}.
The way we do it differs from the way of proving
Theorem~\ref{th:ours}: there will be no need to deal with `long'
words and trying to reduce their lengths. The author hopes to
exhibit the power of the method of wreath recursions.

Before we start the proof, in the next section we give all the
auxiliary notation and lemmas we will need.


\section{Auxiliary Lemmas}

In order to distinguish the states and symbols in
$\mathcal{C}^{\ast}(S)$, we will write $\overline{s}$ to denote the
state correspondent to $s\in S$. If $S=\{s_1,\ldots,s_n\}$, it is
convenient to realize $\overline{s}$ via the \emph{wreath recursion}
(see~\cite{N}):
\begin{equation*}
\overline{s}=\rho_s(\overline{ss_1},\ldots,\overline{ss_n}),
\end{equation*}
where $\rho_s:S\to S$, given by $x\mapsto xs$, corresponds to the
action of $\overline{s}$ on the symbols $S$, and $\overline{ss_i}$
is the state where $\overline{s}$ moves after reading the symbol
$s_i$.

If $x\in S$ and $\alpha\in\mathbf{C^{\ast}}(S)$ then by
$q(\alpha,x)$ we will denote the state, to which $\alpha$ moves
after reading $x\in S$. The transition function of $\alpha$ on $S$
we will denote by $\tau(\alpha)$. So, $\tau(\alpha):S\to S$ and for
all $x\in S$, $x\tau(\alpha)$ is the symbol which outputs
$\mathcal{C}^{\ast}(S)$, reading $x$ in the state $\alpha$. By
definition, for all $a_1,\ldots,a_k,x\in S$, we have
\begin{equation}\label{eq:haha}
q(\overline{a_1}\cdots\overline{a_k},x)=
\overline{a_1x}\cdot\overline{a_2xa_1}\cdots\overline{a_kxa_1\cdots
a_{k-1}}.
\end{equation}
Also the corresponding transition function
$\tau(\overline{a_1}\cdots\overline{a_k})$ is
$\rho_{a_1}\cdots\rho_{a_k}=\rho_{a_1\cdots a_k}$.

\begin{lemma}\label{lm:aux}
Let $S$ be a finite semigroup. Then for all $s,t\in S$,
$\overline{s}=\overline{t}$ in $\C(S)$ if and only if
$\rho_s=\rho_t$.
\end{lemma}

\begin{proof}
Let $s,t\in S$. Then $\overline{s}=\overline{t}$ if and only if
$\rho_s=\rho_t$ and $\overline{sx}=\overline{tx}$ for all $x\in S$.
Recursing the latter, we obtain that $\overline{s}=\overline{t}$ if
and only if $\rho_s=\rho_t$ and $\rho_{sx}=\rho_{tx}$ for all $x\in
S$. It remains to notice that if $\rho_s=\rho_t$, then
$\rho_{sx}=\rho_{tx}$ for all $x\in S$.
\end{proof}

In the following three lemmas we collect some information about
$\C(S)$ for specific types of $S$.

\begin{lemma}\label{lm:easy-to-death-or-even-easier}
Let $G$ be a non-trivial finite group. Then $\mathbf{C^{\ast}}(G)$
is infinite.
\end{lemma}

\begin{proof}
Take any non-identity element $a\in G$. Then $H=\langle a\rangle$ is
a non-trivial commutative group. Obviously the restrictions of the
action of the state $\overline{h}$, $h\in H$, to $H^{\infty}$, is
the same as the action of $\overline{h}$ in $\mathbf{C^{\ast}}(H)$.
Notice that $\mathcal{C}(H)=\mathcal{C}^{\ast}(H)$ and so, by
Theorem~\ref{th:Cayley-groups}, $\overline{H}$ is a free system.
Thus $\mathbf{C^{\ast}}(G)$ is infinite.
\end{proof}

\begin{lemma}\label{lm:something}
Let $L$ be a finite left zero semigroup and $S$ be a finite
semigroup. Then $\mathbf{C^{\ast}}(S\times L)=\mathbf{C^{\ast}}(S)$.
\end{lemma}

\begin{proof}
Let $s\in S$ and $r,t\in L$. Then it follows from Lemma~\ref{lm:aux}
that $\overline{(s,r)}=\overline{(s,t)}$ in $\C(S\times L)$. Hence
$\C(S\times L)$ coincides with $T=\{\overline{(s,r_0)}:s\in S\}$ for
any fixed $r_0\in L$. It is now easy to check that
$\overline{(s,r_0)}\mapsto\overline{s}$ gives rise to an isomorphism
from $T$ onto $\C(S)$.
\end{proof}

\begin{lemma}\label{lm:again-right}
Let $R$ be a finite right zero semigroup with $|R|>1$. Then
$\mathbf{C^{\ast}}(R)$ is a free semigroup of rank $|R|$.
\end{lemma}

\begin{proof}
Let $R=\{q_1,\ldots,q_n\}$ with $n>1$. A direct calculation shows
that $\mathcal{C}^{\ast}(R)$ is the automaton $\mathcal{A}$ of the
type:

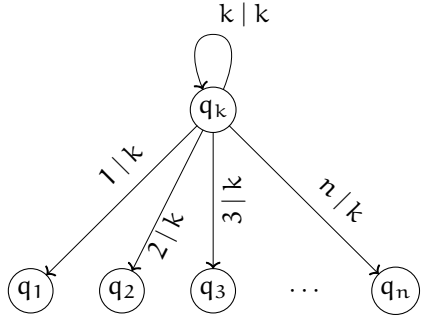
\begin{figure}[htp]
\centering
\includegraphics{cm_1auto.3}
\end{figure}

\vspace{\baselineskip}

We will prove now that the automaton semigroup generated by the
automaton $\mathcal{A}$ is free of rank $n$. If $\mathcal{A}$ has
just read a symbol $i$, then it has just entered state $q_i$ and its
\emph{next} output symbol will be $i$. Thus the action of $q_i$ is
to send a sequence $\alpha$ to $i\alpha$: the sequence $\alpha$ is
shifted right by one symbol and the symbol $i$ is inserted at the
start. So if $w=q_{i_1}\cdots q_{i_m}$ with $i_j\in\{1,\ldots,n\}$,
then
\begin{equation*}
1^{\infty}\cdot w=i_mi_{m-1}\cdots
i_2i_11{^\infty}\quad\mbox{and}\quad 2^{\infty}\cdot
w=i_mi_{m-1}\cdots i_2i_12^{\infty}.
\end{equation*}
The common prefix of $1^{\infty}\cdot w$ and $2^{\infty}\cdot w$
thus determines $w$ and so the automaton semigroup generated by
$\mathcal{A}$ is free with the basis $R$.
\end{proof}

The following lemma is standard to prove, we include it for
completeness.

\begin{lemma}\label{lm:D-class}
Let $S$ be a finite semigroup and let $a,b\in S$. If all $a$, $b$
and $ab$ belong to the same $\mathcal{D}$-class of $S$, then $ab\in
R_{a}\cap L_{b}$.
\end{lemma}

\begin{proof}
By~\cite[Theorem~2.4]{CP} we have that $L_aR_b$ is contained in the
same $\mathcal{D}$-class of $S$. Take any element $c\in R_a\cap
L_b$. Then $c^2\mathcal{D}c$. Let $\mu:S\to\mathcal{T}_X$ be any
faithful representation of $S$ into $\mathcal{T}_X$, the
transformation semigroup on $X$, for some \emph{finite} $X$.
Obviously $H_c$ is a group $\mathcal{H}$-class in $S$ if and only if
$H_{\mu(c)}$ is a group $\mathcal{H}$-class in $\mathcal{T}_X$. We
have that $\mu(c)$ and $\mu(c)^2=\mu(c^2)$ are
$\mathcal{D}$-equivalent in $\mathcal{T}_X$. This happen only if
$\mu(c)\mathcal{H}\mu(c)^2$ in $\mathcal{T}_X$. So that $H_{\mu(c)}$
is a group and so $H_c$ is a group in $S$. Therefore
$c\mathcal{H}c^2$ and so $L_{a}\cap R_{b}$ has an idempotent. By
Clifford-Miller Theorem it now follows that $ab\in R_a\cap L_b$.
\end{proof}


\section{Proof of Theorem~\ref{th:dual-cute}}\label{sec:finiteness}

\begin{proof}[Proof of Theorem~\ref{th:dual-cute}]
First notice that the condition that $S$ does not contain
non-trivial right zero subsemigroups is equivalent to the condition
that there are no two distinct idempotents $e,f\in S$ such that
$e\mathcal{R}f$.

\vspace{\baselineskip}

$(\Rightarrow)$. Suppose that $\mathbf{C^{\ast}}(S)$ is finite.

\vspace{\baselineskip}

{\bf\noindent $S$ is $\mathcal{H}$-trivial.}

Take any $\mathcal{H}$-class $H$ in $S$. With the seek of a
contradiction, suppose that $|H|>1$. Let $T=\{t\in S:Ht\subseteq
H\}$. Then for every $t\in T$, by~\cite[Lemma~2.21]{CP}, the mapping
$\gamma_t:h\mapsto ht$, $h\in H$, is a bijection of $H$ onto itself.
The set of all these bijections forms the so-called
\emph{Sch\"{u}tzenberger group} $\Gamma(H)$ of $H$.
By~\cite[Theorem~2.22]{CP} we have $|\Gamma(H)|=|H|$. Take arbitrary
$t_1,\ldots,t_k\in T$. Then, by~\eqref{eq:haha}, for all $x\in H$ we
have:
\begin{equation*}
q(\overline{t_1}\cdots\overline{t_k},x)=
\overline{t_1x}\cdot\overline{t_2xt_1}\cdots\overline{t_kxt_1\cdots
t_{k-1}}.
\end{equation*}
We also have
$\tau(\overline{t_1}\cdots\overline{t_k})=\tau(\overline{t_1\cdots
t_k})$.

Take now
$\gamma_{t_1},\ldots,\gamma_{t_k},\gamma_x\in\C(\Gamma(H))$. Then
\begin{equation*}
q(\overline{\gamma_{t_1}}\cdots\overline{\gamma_{t_k}},\gamma_{x})=
\overline{\gamma_{t_1}\gamma_{x}}\cdot\overline{\gamma_{t_2}
\gamma_{x}\gamma_{t_1}}\cdots\overline{\gamma_{t_k}\gamma_{x}\gamma_{t_1}\cdots
\gamma_{t_{k-1}}}
\end{equation*}
and $\tau(\overline{\gamma_{t_1}}\cdots\overline{\gamma_{t_k}})=
\tau(\overline{\gamma_{t_1}\cdots\gamma_{t_k}})$. Notice also that
$\gamma_x\gamma_y=\gamma_{xy}$ for $x,y\in T$.

Take $t\in T$ and consider the restriction of $\overline{t}$ to
$H^{\infty}$. From the very definition of $\Gamma(H)$, it follows
that the mapping
$\overline{t}\upharpoonright_{H^{\infty}}\mapsto\overline{\gamma_t}$
gives rise to a well-defined homomorphism from
$\langle\overline{t}\upharpoonright_{H^{\infty}}:t\in T\rangle$ onto
$\mathbf{C^{\ast}}(\Gamma(H))$. By
Lemma~\ref{lm:easy-to-death-or-even-easier}, it means that
$\langle\overline{T}\rangle$ is infinite. Thus
$\mathbf{C^{\ast}}(S)$ is infinite, a contradiction.

\vspace{\baselineskip}

{\bf\noindent There are no distinct idempotents $e,f\in S$ such that
$e\mathcal{R}f$.}

Suppose the converse: that there exist two idempotents $e\neq f$
with $e\mathcal{R}f$. Then $\{e,f\}$ is a $2$-element right zero
semigroup. Restricting the action of $\overline{e}$ and
$\overline{f}$ to $\{e,f\}^{\infty}$ yields the automaton
$\mathcal{C}^{\ast}(\{e,f\})$. By Lemma~\ref{lm:again-right}, it
follows now that $\langle\overline{e},\overline{f}\rangle$ is a free
semigroup of rank $2$, a contradiction.

\vspace{\baselineskip}

$(\Leftarrow)$. We will prove by induction on $|S|$ that if $S$ is
$\mathcal{H}$-trivial and every its $\mathcal{R}$-class contains at
most one idempotent, then $\C(S)$ is finite. The base case $|S|=1$
is obvious. Assume that we have proved this for all such semigroups
of size $\leq n$. Take now any semigroup $S$ with the assumption of
sufficiency of the theorem such that $|S|=n+1$.

Let $M$ be the set of all maximal $\mathcal{D}$-classes in $S$, and
let $I$ be the complement of all these $\mathcal{D}$-classes in $S$.
Suppose first that $I=\varnothing$. Then $S$ consists of a single
$\mathcal{D}$-class and so $S$ is a Rees matrix semigroup
$\mathcal{M}[G;X,Y;P]$ for some group $G$ and a $Y\times X$-matrix
$P=(p_{yx})_{Y\times X}$. Since $S$ is $\mathcal{H}$-trivial, we
have that $G=\{1\}$. Moreover, by~\cite[Theorem~3.4.2]{H}, we may
assume that there exists a column in $P$ consisting entirely of the
element $1$. Since columns in $P$ correspond to
$\mathcal{R}$-classes in $S$, we obtain that $|X|=1$. So, again
by~\cite[Theorem~3.4.2]{H}, it now holds that $S=Y$ is a left zero
semigroup. Then by Lemma~\ref{lm:something}, $\C(S)$ is trivial.

So in the remainder of the proof we may assume that
$I\neq\varnothing$. Obviously $I$ is an ideal in $S$. We prove that
$\C(S)=\langle\overline{S}\rangle$ is finite in the following four
steps.

\vspace{\baselineskip}

{\bf\noindent Step 1: $\mathbf{\langle\overline{I}\rangle}$ is
finite.}

It suffices to prove that there are finitely many products
$\overline{i}\cdot\overline{i_1}\cdots\overline{i_k}\in\langle\overline{I}\rangle$
for any fixed $i\in I$. We have that
$\overline{i}\cdot\overline{i_1}\cdots\overline{i_k}$ and
$\overline{i}\cdot\overline{j_1}\cdots\overline{j_n}$ are distinct
if and only if the restrictions of
$\overline{i_1}\cdots\overline{i_k}$ and
$\overline{j_1}\cdots\overline{j_n}$ on $S^{\infty}\overline{i}$
coincide. Notice that $S^{\infty}\overline{i}\subseteq I^{\infty}$.
Obviously $\overline{i_1}\cdots\overline{i_k}$ and
$\overline{j_1}\cdots\overline{j_n}$ act on $I^{\infty}$ in the same
way as the correspondent products from $\C(I)$ do. Now the claim of
Step 1 follows from the induction hypothesis.

\vspace{\baselineskip}

{\bf\noindent Step 2:
$\mathbf{\overline{I}\langle\overline{S}\rangle^1=
\overline{I}\cup\overline{I}\langle\overline{S}\rangle}$
is finite.}

Take a typical element
$w=\overline{i}\cdot\overline{a_1}\cdots\overline{a_k}\in
\overline{I}\langle\overline{S}\rangle$. Then for all $x\in S$, we
have
\begin{equation*}
q(\overline{i}\cdot\overline{a_1}\cdots\overline{a_k},x)=
\overline{ix}\cdot\overline{a_1xi}\cdot\overline{a_2xia_1}\cdots
\overline{a_kxia_1\cdots a_{k-1}}.
\end{equation*}
Having that $I$ is an ideal in $S$, we deduce that
$q(\overline{i}\cdot\overline{a_1}\cdots\overline{a_k},x)
\in\langle\overline{I}\rangle$. Therefore, having that
$\tau(w)=\rho_{j}$ for $j=i\cdot a_1\cdots a_k\in I$, we obtain
$|\overline{I}\langle\overline{S}\rangle|\leq
|I|\cdot|\langle\overline{I}\rangle|^{|S|}$. Step 2 now follows
immediately.

\vspace{\baselineskip}

{\bf\noindent Step 3:
$\mathbf{\langle\overline{S}\setminus\overline{I}\rangle}$ is
finite.}

Take $a_1,\ldots,a_k\in S\setminus I$. For all $x\in S$, we
have
\begin{equation*}
q_x=q(\overline{a_1}\cdots\overline{a_k},x)=
\overline{a_1x}\cdot\overline{a_2xa_1}\cdots\overline{a_kxa_1\cdots
a_{k-1}}.
\end{equation*}
Obviously, to prove Step 3, it suffices to show that there are only
finitely many expressions $q_x$ for all $a_1,\ldots,a_k\in
S\setminus I$ and $x\in S$. Moreover, by Step 2, it even suffices to
show finiteness of the set $P_0$ of all such expressions $q_x$ with
additional requirement that $a_1x\in S\setminus I$ and $a_2xa_1\in
S\setminus I$. (Otherwise $a_2xa_1\in I$ and so $q_x$ comes from the
finite set
$\overline{S}\cdot\overline{I}\langle\overline{S}\rangle^1$.)

Observe that if $a_1,\ldots,a_i,x$ do not come from the same
$\mathcal{D}$-class from $M$, for some $i$, then $a_ixa_1\cdots
a_{k-1}\in I$. In particular, if $a_1x$ and $a_1$ are not elements
from the same $\mathcal{D}$-class in $M$, then $a_1x\in I$ and so
$q_x\in\overline{I}\langle\overline{S}\rangle^1$.

So, take $x\in S$ such that $a_1x\in S\setminus I$ and $a_2xa_1\in
S\setminus I$. Find the maximum number $m$ such that $a_ixa_1\cdots
a_{i-1}\in S\setminus I$ for all $i\leq m$. Then, by the preceding
paragraph, all $x,a_1,\ldots,a_m$ are from the same
$\mathcal{D}$-class in $M$. Now, $a_{m+1}xa_1\cdots a_m\in I$ and
so, by Step 2, to prove Step 3, it suffices to prove finiteness of
the set $P_1\subseteq P_0$:
\begin{equation*}
P_1=\{\overline{a_1x}\cdot\overline{a_2xa_1}
\cdots\overline{a_kxa_1\cdots a_{k-1}}\mid~a_ixa_1\cdots a_{i-1}\in
S\setminus I~\mbox{for all}~i\leq k\}.
\end{equation*}
Take a typical product
$\overline{a_1x}\cdot\overline{a_2xa_1}\cdots\overline{a_kxa_1\cdots
a_{k-1}}\in P_1$. Then as we discussed above, $x,a_1,\ldots,a_k\in
D$ for some $\mathcal{D}$-class $D$ from $M$. Moreover,
$a_1x,\ldots,a_kx\in D$. Hence by Lemma~\ref{lm:D-class} and
Miller-Clifford Theorem, we have that each $\mathcal{H}$-class
$L_{a_1}\cap R_{x},\ldots,L_{a_k}\cap R_x$ contains an idempotent.
Then, by hypothesis, $a_1\mathcal{L}\cdots\mathcal{L}a_k$.

Now, for all $i\leq k-1$, we have $a_ia_{i+1}\in D$ and so
$L_{a_i}\cap R_{a_{i+1}}$ contains an idempotent. Having that
$a_i\mathcal{L}a_{i+1}$ and by $\mathcal{H}$-triviality, we deduce
that $a_{i+1}$ is an idempotent. Then $a_1=a_1a_2=\cdots=a_1\cdots
a_{k-1}$. Hence
\begin{equation*}
q_x=\overline{a_1x}\cdot\overline{a_2xa_1}\cdots\overline{a_kxa_1\cdots
a_{k-1}}=\overline{a_1x}\cdot\overline{a_2xa_1}\cdots\overline{a_kxa_1}.
\end{equation*}
Finally, $a_1x\in D$ and so $xa_1$, being the unique element of
$L_{a_1}\cap R_{x}$, is an idempotent. Then for all $2\leq i\leq k$,
$a_i\mathcal{L}a_1\mathcal{L}xa_1$, and since $a_i$ is an
idempotent, we obtain that $a_ixa_1=a_i$. Therefore
$q_x=\overline{a_1x}\cdot\overline{a_2}\cdots\overline{a_k}$. Recall
also that $a_2,\ldots,a_k$ are $\mathcal{L}$-equivalent idempotents.

So that, to establish Step 3, it suffices to prove finiteness of the
set
\begin{equation*}
P_2=\{\overline{a_1}\cdots\overline{a_k}\mid~a_1,\ldots,a_k~\mbox{are}~\mathcal{L}-\mbox
{equivalent idempotents from}~S\setminus I\}.
\end{equation*}
Take a typical product $\overline{a_1}\cdots\overline{a_k}\in P_2$.
Let $x\in S$ be arbitrary and let
$q_x=q(\overline{a_1}\cdots\overline{a_k},x)$. Then
$q_x=\overline{a_1x}\cdot\overline{a_2xa_1}\cdots\overline{a_kxa_1}$.
Consider now the following three cases:

\vspace{\baselineskip}

{\bf\noindent Case 1: $(a_1x,a_1)\notin\mathcal{D}$.}

Then, as above, there are only at most
$k_1=|\overline{I}\langle\overline{S}\rangle^1|$ many such $q_x$-s.

\vspace{\baselineskip}

{\bf\noindent Case 2: $a_1x\mathcal{D}a_1$ but
$(a_1x,a_1)\notin\mathcal{L}$.}

Let $y\in S$ be arbitrary. Then
\begin{equation*}
q_{x,y}=q(q_x,y)=
\overline{a_1xy}\cdot\overline{a_2xa_1ya_1x}\cdot\overline{a_3xa_1ya_1xa_2xa_1}\cdots.
\end{equation*}

We will prove that if $k>2$, then $a_3xa_1ya_1xa_2xa_1\in I$.
Suppose the converse. Then $a_3xa_1ya_1xa_2xa_1\mathcal{D}a_1$ and
so $a_1x,a_2,a_1xa_2$ all lie in the same $\mathcal{D}$-class
$D_{a_1}$. Then by Lemma~\ref{lm:D-class}, we obtain that
$L_{a_1x}\cap R_{a_2}=\{e\}$ contains an idempotent. But
$a_1\mathcal{L}a_2$ and $(a_1x,a_1)\notin\mathcal{L}$, so that $a_2$
and $e$ are $\mathcal{R}$-equivalent distinct idempotents, a
contradiction.

Therefore $q_{x,y}\in\overline{S}\cup\overline{S}^2\cup
(\overline{S}^2\cdot\overline{I}\langle\overline{S}\rangle^1)$. In
turn, it implies that there at most $k_2$ elements $q_x$, where
$k_2$ depends only on $S$.

\vspace{\baselineskip}

{\bf\noindent Case 3: $a_1x\mathcal{L}a_1$.}

Then in particular we have $x\mathcal{L}a_1$ and $x$ is an
idempotent. Hence $a_1x=a_1$, $a_2xa_1=a_2,\ldots,a_kxa_1=a_k$.
Therefore $q_x=\overline{a_1}\cdots\overline{a_k}$.

\vspace{\baselineskip}

So, from Case 3 it follows that $\overline{a_1}\cdots\overline{a_k}$
is uniquely determined by $\tau(\overline{a_1}\cdots\overline{a_k})$
and $q_x$-s with $x\in S$ such that $(a_1x,a_1)\notin\mathcal{L}$.
Thus from Cases 1 and 2 we have
\begin{equation*}
|P_2|\leq |S|\cdot (k_1+k_2)^{|S|}.
\end{equation*}
So, $P_2$ is finite and thus Step 3 is established.

\vspace{\baselineskip}

{\bf\noindent Step 4: $\mathbf{\langle\overline{S}\rangle}$ is
finite.}

Follows from
$\langle\overline{S}\rangle=\overline{I}\langle\overline{S}\rangle^1
\cup\langle\overline{S}\setminus\overline{I}\rangle\cup\langle\overline{S}\setminus\overline{I}\rangle
\overline{I}\langle\overline{S}\rangle^1$, by Steps 2 and 3.
\end{proof}


\end{document}